\documentclass{amsart}
\usepackage{amsfonts}
\usepackage{amsmath}
\usepackage{amssymb}
\usepackage{mathrsfs}

\newtheorem{theorem}{Theorem}[section]
\newtheorem{lemma}[theorem]{Lemma}

\theoremstyle{definition}

\theoremstyle{remark}

\numberwithin{equation}{section}

\begin{document}
\title[]{Unitary
operators in the orthogonal complement of a type $\mathrm{I}$ von
Neumann algebra in a type $\mathrm{II}^{}_{1}$ factor}

\author{Xiaoyan Zhou}
\address{Xiaoyan Zhou}\address{School of Mathematical Sciences, Dalian
               University of Technology. Dalian~{116024}. China}
\email{xiaoyanzhou.oa@gmail.com}

\author{Rui Shi}
\address{Rui Shi}
\address{School of Mathematical Sciences, Dalian
               University of Technology. Dalian~{116024}. China}
\email{ruishi.math@gmail.com} \subjclass[2000]{Primary 46L10;
Secondary 47C15}

\begin{abstract}
 It is well-known that the equality $$L^{}_{G}\ominus
L^{}_{H}=\overline{\mathrm{span}\{L_{g}:g\in G-H\}^{\mathrm{SOT}}}$$
holds for $G$ an i.c.c.~group and $H$ a subgroup in $G$, where
$L^{}_{G}$ and $L^{}_{H}$ are the corresponding group von Neumann
algebras and $L^{}_{G}\ominus L^{}_{H}$  is the set $\{x\in
L^{}_{G}:E^{}_{L^{}_{H}}(x)=0\}$ with $E^{}_{L^{}_{H}}$ the
conditional expectation defined from $L^{}_{G}$ onto $L^{}_{H}$.
Inspired by this, it is natural to ask whether the equality
$$N\ominus A=\overline{\mathrm{span}\{u: u\mbox{ is unitary in
}N\ominus A\}^{\mathrm{SOT}}}$$ holds for $N$ a type
$\mbox{II}^{}_{1}$ factor and $A$ a von Neumann subalgebra of $N$.
In this paper, we give an affirmative answer to this question for
the case $A$ a type I von Neumann algebra.
\\
Key words.~~type $\mathrm{II}_{1}$ factor, type $\mathrm{I}$ von
Neumann algebra, conditional expectation, orthogonal complement,
unitary operator
\end{abstract}
\maketitle
\section{introduction}

Throughout this paper, all Hilbert spaces discussed are
{\textit{complex and separable}}. Let $(N,\tau)$ be a finite von
Neumann algebra with a faithful normal normalized trace $\tau$ and
$A$ be a von Neumann subalgebra of $N$. Then the trace $\tau$
induces an inner product $\langle\cdot,\cdot\rangle$ on $N$ which is
defined by $\langle x,y\rangle=\tau(y^{*}x), \forall x,y\in N$.
Denote by $L^{2}(N)$ (resp.~$L^{2}(A)$) the completion of $N$
(resp.~$A$) with respect to the inner product, then $L^{2}(A)$ is a
subspace of $L^{2}(N)$. Let $e^{}_{A}$ denote the projection from
$L^{2}(N)$ onto $L^{2}(A)$. The trace-preserving conditional
expectation $E^{}_{A}$ of $N$ onto $A$ is defined to be the
restriction $e^{}_{A}|^{}_{N}$. By\cite{A}, $E^{}_{A}$ has the
following properties:
\begin{enumerate}
\item $e^{}_{A}|_{N}=E^{}_{A}$ is a norm reducing map from $N$ onto $A$ with
$E^{}_{A}(1)=1$;
\item the equality $E^{}_{A}(axb)=aE^{}_{A}(x)b$ holds for all $x\in
N$ and $a, b\in A$;
\item
$\tau(xE^{}_{A}(y))=\tau(E^{}_{A}(x)E^{}_{A}(y))=\tau(E^{}_{A}(x)y)$
for all $x,y\in N$;
\item $e^{}_{A}xe^{}_{A}=E^{}_{A}(x)e^{}_{A}=e^{}_{A}E^{}_{A}(x)$
for all $x\in N$.
\end{enumerate}

Let $G$ be a (countable) discrete i.c.c.~group and denote by
$l^{2}(G)$ the Hilbert space of square-summable sequences. Given
every $g$ in $G$, the operator $L^{}_{g}$ is defined by
$(L^{}_{g}x)(h)=x(g^{-1}_{}h)$, for every $x$ in $l^{2}(G)$ and $h$
in $G$. This is a unitary operator. Let $L^{}_{G}$ be the von
Neumann algebra generated by $\{L^{}_{g}:g\in G\}$. It is well-known
that $L^{}_{G}$ is a type $\mbox{II}^{}_{1}$ factor. For a subgroup
$H$ in $G$, define $$L^{}_{G}\ominus L^{}_{H}\triangleq\{x\in
L^{}_{G}:E^{}_{L^{}_{H}}(x)=0\}.$$ Thus we obtain that
$$L^{}_{G}\ominus
L^{}_{H}=\overline{\mathrm{span}\{L_{g}:g\in
G-H\}^{\mathrm{SOT}}}.$$ Inspired by this, it is natural to ask
whether the equality
$$N\ominus A=\overline{\mathrm{span}\{u: u\mbox{ is unitary in
}N\ominus A\}^{\mathrm{SOT}}}$$ holds for $N$ a type
$\mbox{II}^{}_{1}$ factor and $A$ a von Neumann subalgebra of
$N$. In this paper, we give an affirmative answer to this question for the
case $A$ a type $\mbox{I}$ von Neumann algebra in Theorem $2.6$.

\section{proofs}

In this paper, the matrix representations of operators will be used
frequently. We briefly recall the relationship between conditional
expectations with respect to matrix representations of operators.
Let $e^{}_{1},\ldots, e^{}_{n}\in N$ be mutually
equivalent orthogonal projections such that $\sum^{n}_{i=1}e^{}_{i}=1$, where
$1$ is the identity of $N$. Then for every $x\in N$, we can
express $x$ in the form $$x=\begin{pmatrix}
x^{}_{11}&\cdots&x^{}_{1n}\\
\vdots&\ddots&\vdots\\
x^{}_{n1}&\cdots&x^{}_{nn}\\
\end{pmatrix}
\begin{matrix}
\mbox{ran }e^{}_{1}\\
\vdots\\
\mbox{ran }e^{}_{n}\\
\end{matrix}$$
and there exists a $*$-isomorphism $\varphi$ from $N$ onto
$\mathbb{M}_{n}(N^{}_{e^{}_{1}})$, where $N^{}_{e^{}_{1}}$ is the
restriction of $e^{}_{1}Ne^{}_{1}$ on $\mathrm{ran}~e^{}_{1}$ and
denote by $\mathbb{M}_{n}(N^{}_{e^{}_{1}})$ the set $n$-by-$n$
matrices with entries in $N^{}_{e^{}_{1}}$.  On the other hand, let
$\tau$ be a faithful normal normalized trace on $N$, and the trace
$\tau^{}_{n}$ on $\mathbb{M}_{n}(N)$ is defined by
${\tau}^{}_{n}({x})=\frac{1}{n}(\sum^{n}_{i=1}\tau(x^{}_{ii}))$,
where $x$ in $\mathbb{M}_{n}(N)$ is of the form
$$x=\begin{pmatrix}
x^{}_{11}&\cdots&x^{}_{1n}\\
\vdots&\ddots&\vdots\\
x^{}_{n1}&\cdots&x^{}_{nn}\\
\end{pmatrix}$$ and $x^{}_{ij}$ is in $N$ for $i,j=1,\ldots,n$. We observe that
$\tau^{}_{n}$ is a faithful normal normalized trace. For a von
Neumann subalgebra $A$ in $N$, there exist conditional expectations
$E^{}_{A}$ from $N$ onto $A$ and $E^{}_{\mathbb{M}_{n}(A)}$ from
$\mathbb{M}_{n}(N)$ onto $\mathbb{M}_{n}(A)$. Given fixed $i^{}_{0}$
and $j^{}_{0}$, let $x^{}_{i^{}_{0}j^{}_{0}}$ denote again the
operator in $\mathbb{M}_{n}(N)$ with all entries $0$ but the
$(i^{}_{0},j^{}_{0})$ entry $x^{}_{i_{0}j_{0}}$. By the fact that
$E^{}_{\mathbb{M}_{n}(A)}$ is ${\mathbb{M}_{n}(A)}$-modular, the
equality
$$E^{}_{\mathbb{M}_{n}(A)}(x^{}_{i^{}_{0}j^{}_{0}})=
E^{}_{\mathbb{M}_{n}(A)}(e^{}_{i^{}_{0}}x^{}_{i^{}_{0}j^{}_{0}}e^{}_{j^{}_{0}})=
e^{}_{i^{}_{0}}E^{}_{\mathbb{M}_{n}(A)}(x^{}_{i^{}_{0}j^{}_{0}})e^{}_{j^{}_{0}}$$
ensures that all but the $(i^{}_{0},j^{}_{0})$ entry of
$E^{}_{\mathbb{M}_{n}(A)}(x^{}_{i^{}_{0}j^{}_{0}})$ are $0$, where
$e^{}_{i}$ is the diagonal projection with all entries $0$ except
the $(i,i)$ one being the identity of $N$. Therefore
$x^{}_{i^{}_{0}j^{}_{0}}$ is in
$\mathbb{M}_{n}(N)\ominus\mathbb{M}_{n}(A)$ if and only if the
$(i^{}_{0},j^{}_{0})$ entry of $x^{}_{i^{}_{0}j^{}_{0}}$ is in
$N\ominus A$.

In what follows, $N$ will always denote a von Neumann algebra. Every
subalgebra of $N$ we consider here is self-adjoint, weakly closed
and contains the unit $1$ of $N$. For a subset $\mathscr{S}\subseteq
N$, denote by $\mathscr{U}(\mathscr{S})$ the unitary operators in
$\mathscr{S}$.
\begin{lemma}
Let $N$ be a von Neumann algebra and
$$M_{n}=\{x:x\in \mathbb{M}_{n}(N),x^{}_{ii}=0, i=1,\ldots,n\},$$ then
$M_{n}=\mathrm{span}\{u:u\in\mathscr{U}(M_{n})\}$.
\end{lemma}
\begin{proof}
For each $x$ in $M_{n}$, we can write $x$ in the form
$$x=\begin{pmatrix}
x^{}_{11}&x^{}_{12}&\cdots&x^{}_{1n}\\
x^{}_{21}&x^{}_{22}&\cdots&x^{}_{2n}\\
\vdots&\vdots&\ddots&\vdots\\
x^{}_{n1}&x^{}_{n2}&\cdots&x^{}_{nn}\\
\end{pmatrix}^{}_{n\times n},$$ where $x^{}_{ij}$ is in $N$ and
$x^{}_{ii}=0$,
for $i,j=1,\ldots,n$. Without loss of generality, we may assume
$x^{}_{i^{}_{0}j^{}_{0}}\neq 0,i^{}_{0}<j^{}_{0}$ and all other
entries $0$. Thus we can write $x$ in the form of block matrix
$$x=\begin{pmatrix}
X^{}_{11}&X^{}_{12}\\
X^{}_{21}&X^{}_{22}\\
\end{pmatrix},$$ where $X^{}_{12}$ is a $k$-by-$k$ matrix for some
$k\leq n-1$ with $x^{}_{i^{}_{0}j^{}_{0}}$ on the main diagonal of
$X_{12}^{}$.

Note that
$x^{}_{i^{}_{0}j^{}_{0}}=\sum^{4}_{i=1}\lambda¡­^{}_{i}u^{}_{i}$ for
some $u^{}_{i}\in \mathscr{U}(N)$ and $\lambda^{}_{i}\in\mathbb{C}$.
For the sake of simplicity, we can write $X_{12}$ in the form
$$x^{}_{i^{}_{0}j^{}_{0}}\oplus 0^{(k-1)}
=\sum^{4}_{i=1}\frac{\lambda^{}_{i}}{2}(u^{}_{i}\oplus
v^{(k-1)}+u^{}_{i}\oplus (-v)^{(k-1)}),$$ where $v$ is unitary in
${N}$. Write $\widetilde{u}^{}_{i}=u^{}_{i}\oplus v^{(k-1)}$ and
$\widehat{u}^{}_{i}=u^{}_{i}\oplus (-v)^{(k-1)}$, then $x$ can be
written in the form
$$x=\sum^{4}_{i=1}\frac{\lambda^{}_{i}}{2}\left(\begin{pmatrix}
0&\widetilde{u}^{}_{i}\\
I^{}_{X^{}_{21}}&0\\
\end{pmatrix}+\begin{pmatrix}
0&\widehat{u}^{}_{i}\\
-I^{}_{X^{}_{21}}&0\\
\end{pmatrix}\right),$$ where $I^{}_{X^{}_{21}}$ is the
identity of $\mathbb{M}^{}_{n-k}({N})$.

By a similar method, every $x$ in $M_{n}$ can be written as a linear
combination of finitely many unitary operators. Thus we finish the
proof.
\end{proof}

\begin{lemma}
If $N$ is a von Neumann algebra, $A\subseteq N$ is a von Neumann
subalgebra and $N\ominus
A=\overline{\mathrm{span}\{u:u\in\mathscr{U}(N\ominus
A)\}^{\mathrm{SOT}}}$, then
$$\mathbb{M}_{n}(N\ominus A)=
\overline{\mathrm{span}\{u:u\in\mathscr{U}( \mathbb{M}_{n}(N\ominus
A))\}^{\mathrm{SOT}}}.$$
\end{lemma}
\begin{proof}
For each $x\in N\ominus A$, there exists a sequence
$\{x_{n}\}_{n\in\Lambda}$ in $N\ominus A$, such that
$x_{n}\xrightarrow{\mathrm{SOT}}x$,
$x_{n}=\sum^{k^{}_{n}}_{i=1}\lambda_{n_{i}}u_{n_{i}}$, $u_{n_{i}}\in
\mathscr{U}(N\ominus A)$, $\lambda_{n_{i}}\in \mathbb{C}$. Without
loss of generality and for the sake of simplicity, we may assume
$$\tilde{x}=\left(\begin{array}{cccc}
x&0&\ldots&0\\
0&0&\ldots&0\\
\vdots&\vdots&\ddots&\vdots\\
0&0&\ldots&0\\
\end{array}
\right),\quad\tilde{x}_{n}=\left(\begin{array}{cccc}
x_{n}&0&\ldots&0\\
0&0&\ldots&0\\
\vdots&\vdots&\ddots&\vdots\\
0&0&\ldots&0\\
\end{array}
\right).$$ Note that $x$ can be moved to the $(i,j)$ entry by
multiplying $u_{i}$ on the left and $u_{j}$ on the right, where
$u_{i}$ (resp.~$u_{j}$) is the elementary matrix obtained by
swapping row $1$(resp. column $1$) and row $i$ (resp.~column $j$) of
the identity matrix for $1\leq i,j\leq n$.

Then the result
$\tilde{x}\in\overline{\mathrm{span}\{u:u\in\mathscr{U}(N\ominus
A)\}^{\mathrm{SOT}}}$ follows from the two relations
$$\tilde{x}_{n}\xrightarrow{\mathrm{SOT}}\tilde{x},$$
$$
\tilde{x}^{}_{n}= \sum^{k^{}_{n}}_{i=1}\dfrac{\lambda^{}_{n_{i}}}{2}
\left(\left(\begin{array}{cccc}
u_{n_{i}}\\
&v\\
&&\ddots\\
&&&v\\
\end{array}
\right)+\left(\begin{array}{cccc} u_{n_{i}}\\
&-v\\
&&\ddots\\
&&&-v\\
\end{array}\right)\right),$$
where $v$ is a unitary operator in $\mathscr{U}(N\ominus A)$.
\end{proof}

\begin{lemma}
If $N$ is a type $\mathrm{II}^{}_{1}$ factor with a faithful normal
normalized trace $\tau$  and $A \subseteq N$ is a diffuse abelian
von Neumann subalgebra, then $$N\ominus A=\mathrm{span}\{u:u\in
\mathscr{U}(N \ominus A)\}.$$
\end{lemma}
\begin{proof}
Since $N$ is a type $\mathrm{II}^{}_{1}$ factor, there exist four
equivalent mutually orthogonal projections $\{e_{i}\}_{1\leq i\leq
4}\subseteq A$, such that $\sum^{4}_{i=1}e_{i}=1$. Denote by
$M$ the reduced von Neumann algebra $e^{}_{1}Ne^{}_{1}$. Then there exists a $*$-isomorphism
$\varphi$ from $N$ onto $\mathbb{M}_{4}(M)$  so that
$\varphi(A)=\bigoplus^{4}_{i=1}A_{i}$, where $A^{}_{i}$ is a diffuse
abelian von Neumann subalgebra in $M$. For the
sake of simplicity, we assume $N=\mathbb{M}_{4}(M)$,
$A=\bigoplus^{4}_{i=1}A_{i}$.

Denote by $M_{4}=\{(x^{}_{ij})^{}_{1\leq i,j\leq4}\in N:x^{}_{ii}=0~
\textmd{ for  }1\leq i\leq 4 \}$, then
$$M_{4}=\mathrm{span}\{u:u\in\mathscr{U}(M_{4})\}$$ following from
Lemma $2.1$. Thus we only need to prove
\begin{equation}
\bigoplus_{i=1}^{4}M\ominus
A_{i}\subseteq\mathrm{span}\{u:u\in \mathscr{U}(N \ominus
A)\},
\end{equation} since $M_{4}\subseteq N\ominus A$.

Consider the matrix
$$\widetilde{x}=\begin{pmatrix}
x^{}_{1}&0&0&0\\
0& x^{}_{2}&0&0\\
0&0& 0&0\\
0&0&0& 0\\
\end{pmatrix},$$ where $x^{}_{1}=x_{1}^{*}\in M\ominus A_{1}$,
$x^{}_{2}=x_{2}^{*}\in M\ominus A_{2}$, $\| x^{}_{1}\|<1$, $\|
x^{}_{2}\|<1$. Since for each $y\in N$, we have
$y=\lambda^{}_{1}y^{}_{1}+\lambda^{}_{2}y^{}_{2}$, where
$y^{}_{i}=y^{*}_{i}$, $\lambda^{}_{i}\in\mathbb{C}$, $\|
y^{}_{i}\|<1$, for $1\leq i\leq2$. Let
$$\widetilde{u}^{}_{1}=\begin{pmatrix}
x^{}_{1}&0&\sqrt{1-x_{1}^{2}}&0\\
0& x^{}_{2}&0&\sqrt{1-x_{2}^{2}}\\
0&-\sqrt{1-x_{2}^{2}}& 0&x^{}_{2}\\
-\sqrt{1-x_{1}^{2}}&0&x^{}_{1}& 0\\
\end{pmatrix},$$
$$\widetilde{u}^{}_{2}=\begin{pmatrix}
x^{}_{1}&0&\sqrt{1-x_{1}^{2}}&0\\
0&x^{}_{2}&0&\sqrt{1-x_{2}^{2}}\\
0&\sqrt{1-x_{2}^{2}}&0&-x^{}_{2}\\
\sqrt{1-x_{1}^{2}}&0&-x^{}_{1}&0\\
\end{pmatrix},$$
$$\widetilde{u}^{}_{3}=\begin{pmatrix}
0&0&\sqrt{1-x_{1}^{2}}&0\\
0&0&0&\sqrt{1-x_{2}^{2}}\\
0&0&0&0\\
0&0&0&0\\
\end{pmatrix},$$
then we obtain that
\begin{equation}
\widetilde{x}=\tfrac{1}{2}\widetilde{u}^{}_{1}+\tfrac{1}{2}\widetilde{u}^{}_{2}-\widetilde{u}^{}_{3}.
\end{equation} Notice that
$\widetilde{u}^{}_{1},\widetilde{u}^{}_{2}$ are unitary operators in
$N\ominus A$ and $\widetilde{u}^{}_{3}$ belongs to $M_{4}$, then
$(2.2)$ and Lemma $2.1$ allow us to conclude that $$\tilde{x}\in
\mathrm{span}\{u:u\in \mathscr{U}(N \ominus A)\}.$$

Similarly, we can also show that
$$\begin{pmatrix}
0&0&0&0\\
0&0&0&0\\
0&0&x^{}_{3}&0\\
0&0&0&x^{}_{4}\\
\end{pmatrix}
$$ is a linear combination of finitely
many unitary operators in $N\ominus A$, where $x^{}_{3}\in M\ominus
A_{3}$, $x^{}_{4}\in M\ominus A_{4}$, thus we finish the proof of
$(2.1)$.
\end{proof}

\begin{lemma}
If $N$ is a type $\mathrm{II}^{}_{1}$ factor with a faithful normal
normalized trace $\tau$ and $A \subseteq N$ is an atomic abelian von
Neumann subalgebra, then $$N\ominus
A=\overline{\mathrm{span}\{u:u\in \mathscr{U}(N \ominus
A)\}^{\mathrm{SOT}}}.$$
\end{lemma}
\begin{proof}
We now consider four cases respectively.

\noindent $(\rm{i})$ $A=\mathbb{C}1$.

Since $N$ is a type $\mathrm{II}^{}_{1}$ factor, there exist two
equivalent mutually orthogonal projections $p$ and $q$ in $N$ such
that $p+q=1$. Denote by $M$ the reduced von Neumann algebra $pNp$
with a faithful normal normalized trace $\tau^{}_{M}$. Then there
exists a $*$-isomorphism $\varphi$ from $N$ onto $\mathbb{M}_{2}(M)$
so that $\varphi(A)=\mathbb{C}p_{}^{(2)}$. If we write
$\widetilde{N}=\mathbb{M}_{2}(M )$, $\widetilde{A}=\varphi(A)$, then
we obtain
$$\widetilde{N}\ominus\widetilde{A}=\{\tilde{x}\in\widetilde{N}:
\tau^{}_{M}(x_{11}+x_{22})=0,\tilde{x}=(x_{ij})_{1\leq i,j\leq 2}\}.$$
Note that
\begin{equation}
\begin{pmatrix}
x^{}_{11}&x^{}_{12}\\
x^{}_{21}&x^{}_{22}\\
\end{pmatrix}=\begin{pmatrix}
x^{}_{11}+x^{}_{22}&0\\
0&0\\
\end{pmatrix}+\begin{pmatrix}
-x^{}_{22}&0\\
0&x^{}_{22}\\
\end{pmatrix}+\begin{pmatrix}
0&x^{}_{12}\\
x^{}_{21}&0\\
\end{pmatrix}.\end{equation}
For each $x^{}_{22}\in M$,
$x^{}_{22}=\sum_{i=1}^{4}\lambda^{}_{i}u^{}_{i}$, where
$u^{}_{1},\ldots,u^{}_{4}$ are unitary operators in $M$  so that
$$\begin{pmatrix}
-x^{}_{22}&0\\
0&x^{}_{22}\\
\end{pmatrix}=\sum_{i=1}^{4}\lambda^{}_{i}\begin{pmatrix}
-u^{}_{i}&0\\
0&u^{}_{i}\\
\end{pmatrix}.$$
Since $\begin{pmatrix}
-u^{}_{i}&0\\
0&u^{}_{i}\\
\end{pmatrix}\in \widetilde{N}\ominus\widetilde{A}$, we have
\begin{equation}
\begin{pmatrix}
-x^{}_{22}&0\\
0&x^{}_{22}\\
\end{pmatrix}\in \mathrm{span}\{u:u\in\mathscr{U}
(\widetilde{N}\ominus\widetilde{A})\}.
\end{equation}
Denote by
$$M_{2}=\left\{\begin{pmatrix}
0&x^{}_{12}\\
x^{}_{21}&0\\
\end{pmatrix}:x^{}_{12},x^{}_{21}\in M\right\}.$$
Since $M_{2}\subseteq \widetilde{N}\ominus\widetilde{A}$, by Lemma
$2.1$ we obtain
\begin{equation}
M_{2}\subseteq\mathrm{span}
\{u:u\in\mathscr{U}(\widetilde{N}\ominus\widetilde{A})\}.
\end{equation}

For $x\in M$ and $\tau^{}_{M}(x)=0$, we may assume $x=x^{*}$,
$\|x\|< 1$, since
$$x=\tfrac{x+x_{}^{*}}{2}+i\tfrac{x-x_{}^{*}}{2i} \mbox{~~and~~}
\tau^{}_{M}(\tfrac{x+x_{}^{*}}{2})=\tau^{}_{M}(\tfrac{x-x_{}^{*}}{2i})=0.$$
Let
$$\tilde{x}=\begin{pmatrix}
x&0\\
0&0\\
\end{pmatrix},$$
$$\tilde{u}^{}_{1}=\begin{pmatrix}
x&\sqrt{1-x^{2}}\\
-\sqrt{1-x^{2}}&x\\
\end{pmatrix}, \tilde{u}^{}_{2}=\begin{pmatrix}
x&\sqrt{1-x^{2}}\\
\sqrt{1-x^{2}}&-x\\
\end{pmatrix},\tilde{u}^{}_{3}=\begin{pmatrix}
0&\sqrt{1-x^{2}}\\
0&0\\
\end{pmatrix},$$
then we have
\begin{equation}
\tilde{x}=\tfrac{1}{2}\tilde{u}^{}_{1}+
\tfrac{1}{2}\tilde{u}^{}_{2}-\tilde{u}^{}_{3}.
\end{equation}
Observe that $\tilde{u}^{}_{1},\tilde{u}^{}_{2}$ are both unitary in
$\widetilde{N}\ominus\widetilde{A}$ and $\tilde{u}^{}_{3}\in M_{2}$.

Hence
$\widetilde{N}\ominus\widetilde{A}=\mathrm{span}\{u:u\in\mathscr{U}(\widetilde{N}\ominus\widetilde{A})\}$
follows from $(2.3),(2.4),(2.5)$ and $(2.6)$.

\noindent $(\rm{ii})$ $A=\mathbb{C}p+\mathbb{C}q$, where $p$ and $q$ are
mutually orthogonal projections in $N$ with sum $1$.

Each $x\in$ $N\ominus A$ can be written as
$$x=\begin{pmatrix}
x^{}_{11}&x^{}_{12}\\
x^{}_{21}&x^{}_{22}\\
\end{pmatrix}
$$
with respect to the decomposition $1=p+q$, where $x^{}_{11}\in
p(N\ominus A)p$, $x^{}_{12}\in pNq$, $x^{}_{21}\in qNp$,
$x^{}_{22}\in q(N\ominus A)q$. By $\rm{(i)}$, we obtain that
$$pNp\ominus \mathbb{C}p=\mathrm{span}\{u:u\in\mathscr{U}(pNp\ominus
\mathbb{C}p)\},$$
$$qNq\ominus\mathbb{C}q=\mathrm{span}
\{u:u\in\mathscr{U}(qNq\ominus\mathbb{C}q)\},$$ therefore $$\begin{pmatrix}
x^{}_{11}&0\\
0&x^{}_{22}\\
\end{pmatrix}\in \mathrm{span}\{u:u\in\mathscr{U}(N\ominus A)\}.$$
We only need to prove
\begin{equation}
\begin{pmatrix}
0&x^{}_{12}\\
x^{}_{21}&0\\
\end{pmatrix}\in \mathrm{span}\{u:u\in\mathscr{U}(N\ominus
A)\}.\end{equation}

If $\tau(p)$ is rational, then we assume
$$\frac{\tau(p)}{\tau(q)}=\frac{m}{n},\mbox{ for some } m, n \in\mathbb{N^{+}}.$$
Let $\{p^{}_{i}\}_{1\leq i\leq m}$ and $\{q^{}_{j}\}_{1\leq j\leq
n}$ be two families of mutually orthogonal projections in $N$ such
that $p^{}_{1}+p^{}_{2}+\cdots+p^{}_{m}=p$,
$q^{}_{1}+q^{}_{2}+\cdots+q^{}_{n}=q$ and
$\tau(p^{}_{i})=\tau(q^{}_{j})$, for all $1\leq i\leq m$, $1\leq j
\leq n$. Denote by $M=p^{}_{1}Np^{}_{1}$ with a faithful normal
normalized trace $\tau^{}_{M}$, then there exists a $*$-isomorphism
$\varphi$  from $N$ onto $\mathbb{M}_{m+n}(M)$ so that
$\varphi(A)=\mathbb{C}p^{(m)}_{1}\oplus \mathbb{C}p^{(n)}_{1}$.
Denote by $\widetilde{N}=\varphi(N)$, $\widetilde{A}=\varphi(A)$,
$$N^{}_{0}=\left\{\begin{pmatrix}
0&x^{}_{12}\\
x^{}_{21}&0\\
\end{pmatrix}: x^{}_{12}\in pNq,~x^{}_{21}\in qNp\right\},$$
$$M_{m+n}=\{\tilde{x}\in \mathbb{M}^{}_{m+n}(M):
x^{}_{ii}=0,\textmd{for }1\leq i\leq m+n\}.$$ Note that
$$\widetilde{N}\ominus\widetilde{A}=\{(x^{}_{ij}) ^{}_{ 1\leq
i,j\leq m+n}\in
\widetilde{N}:\sum^{m}_{i=1}\tau^{}_{M}(x^{}_{ii})=0,\sum^{m+n}_{i=m+1}\tau^{}_{M}(x^{}_{ii})=0\},$$
then $\varphi(N^{}_{0})\subseteq M_{m+n}\subseteq
\widetilde{N}\ominus\widetilde{A}$. Then applying Lemma $2.1$ to
$M_{m+n}$, we obtain that
$$M_{m+n}=\mathrm{span}\{u:u\in\mathscr{U}(M_{m+n})\},$$ so that
$$\varphi(N^{}_{0})\subseteq\mathrm{span}\{u:u\in\mathscr{U}(\widetilde{N}\ominus\widetilde{A})\},$$
thus $(2.7)$ holds.

If $\tau(p)$ is irrational, then let $\{p^{}_{n}\}_{n\in
\Lambda},\{q^{}_{n}\}_{n\in\Lambda}$ be two families of increasing
subprojections of $p$ and $q$ respectively such that
$\tau(p^{}_{n})\rightarrow \tau(p)$, $\tau(q^{}_{n})\rightarrow
\tau(q)$ and for all $n\in \Lambda$, both $\tau(p^{}_{n})$ and
$\tau(q^{}_{n})$ are rational. Thus for $n\in \Lambda$, $x\in N$, we
have
$$p^{}_{n}xq^{}_{n} \xrightarrow{\mathrm{SOT}} pxq.$$ Next we show
that
\begin{equation}
p^{}_{n}xq^{}_{n}\in\mathrm{span}\{u:u\in\mathscr{U}(N\ominus
A)\}. \end{equation}

For each $n\in \Lambda$, suppose
$$\frac{\tau(p_{n})}{\tau(q_{n})}=\frac{k^{}_{n}}{l^{}_{n}} \mbox{ with } k^{}_{n},
l^{}_{n}\in\mathbb{N^{+}}.$$
Let $\{p^{}_{n_{i}}\}_{1\leq i \leq k^{}_{n} }$ and
$\{q^{}_{n_{j}}\}_{1\leq j\leq l^{}_{n}}$ be two families of
mutually orthogonal equivalent projections in $ N $ such that
$p^{}_{n_{1}}+p^{}_{n_{2}}+\cdots+p^{}_{n_{k_{n}}}=p^{}_{n}$,
$q^{}_{n_{1}}+q^{}_{n_{2}}+\cdots+q^{}_{n_{l_{n}}}=q^{}_{n}$. Then
there exists a $*$-isomorphism $\varphi$ from $N$ onto $\varphi(N)$
such that
 $$\varphi((p^{}_{n}+q^{}_{n})N(p^{}_{n}+q^{}_{n}))=\mathbb{M}_{k^{}_{n}+l^{}_{n}}
 (p^{}_{n_{1}}Np^{}_{n_{1}})$$
and $\varphi|_{(1-p_{n}-q_{n})N(1-p_{n}-q_{n})}$ is the identity map.
Denote by
$$N^{}_{n_{0}}=\left\{\begin{pmatrix}
0&x\\
y&0\\
\end{pmatrix}: x\in p^{}_{n}Nq^{}_{n},~y\in q^{}_{n}Np^{}_{n}\right\},$$
$$ N_{1}=((p-p_{n})N(p-p_{n}))\ominus ((p-p_{n})A(p-p_{n})),$$
$$ N_{2}=((q-q_{n})N(q-q_{n}))\ominus ((q-q_{n})A(q-q_{n})),$$
$$M_{k^{}_{n}+l^{}_{n}}=\{\{x_{ij}\}^{}_{ 1\leq i,j\leq
k^{}_{n}+l^{}_{n}}\in
\mathbb{M}^{}_{k^{}_{n}+l^{}_{n}}(p^{}_{n_{1}}Np^{}_{n_{1}}):
x^{}_{ii}=0,\textmd{ for } 1\leq i\leq k^{}_{n}+l^{}_{n}\}.$$

By Lemma $2.1$ and Case $\rm{(i)}$, we have each operator in
$$M_{k^{}_{n}+l^{}_{n}}\oplus N_{1}\oplus N_{2}$$can be written as a
linear combination of finitely many unitary operators in this set.
Note that
$$\varphi(N^{}_{n_{0}})\subseteq M_{k^{}_{n}+l^{}_{n}}\mbox{  and  }
M_{k^{}_{n}+l^{}_{n}}\oplus N_{1}\oplus N_{2}\subseteq
\varphi(N)\ominus\varphi(A),$$ so that
$$\varphi(N^{}_{n_{0}})\oplus
N_{1}\oplus
N_{2}\subseteq\mathrm{span}\{u:u\in\mathscr{U}(\varphi(N)\ominus\varphi(A))\},$$
thus $(2.8)$ holds.

\noindent $\rm{(iii)}$ $A=\sum_{i=1}^{n}\mathbb{C}p^{}_{i}$, where
$\{p^{}_{i}\}_{1\leq i\leq n}$ is a family of mutually orthogonal
projections in $N$ with sum $1$.

Each $x\in$ $N\ominus A$ can be written as
$$x=\begin{pmatrix}
x^{}_{11}&x^{}_{12}&\ldots&x^{}_{1n}\\
x^{}_{21}&x^{}_{22}&\ldots&x^{}_{2n}\\
\vdots&\vdots&\ddots&\vdots\\
x^{}_{n1}&x^{}_{n2}&\ldots&x^{}_{nn}\\
\end{pmatrix}$$
with respect to the decomposition $1=\sum_{1\leq i\leq n}p^{}_{i}$,
where $x^{}_{ii}\in p^{}_{i}Np^{}_{i}\ominus \mathbb{C} p^{}_{i}$,
$x^{}_{ij}\in  p^{}_{i}N p^{}_{j}$, for $1\leq i\neq j\leq n$.

For $1\leq i<j \leq n$, denote by
$$P_{ij}=\bigoplus_{k\neq i,j;1 \leq k\leq n} p^{}_{k}Np^{}_{k}\ominus\mathbb{C}p^{}_{k},$$
$$N_{ij}=(p^{}_{i}+p^{}_{j})(N\ominus A)(p^{}_{i}+p^{}_{j}),$$ then
$$P_{ij}=\mathrm{span}\{u:u\in
\mathscr{U}( P_{ij})\},~~~~N_{ij}=\overline{\mathrm{span}\{u:u\in
\mathscr{U}( N_{ij})\}^{\mathrm{SOT}}}$$ following from Case
$\rm{(i)}$ and Case $\rm{(ii)}$, therefore
$$N_{ij}\oplus P_{ij}=\overline{\mathrm{span}\{u:u\in
\mathscr{U}(N_{ij}\oplus P_{ij})\}^{\mathrm{SOT}}}.$$

\noindent $\rm{(iv)}$  $A=\sum^{\infty}_{i=1}\mathbb{C}p^{}_{i}$,
where $\{p^{}_{i}\}^{\infty}_{i=1}$ is a family of  mutually
orthogonal projections in $A$ with sum $1$.

Denote by
$$q^{}_{i}=\sum^{i}_{k=1} p^{}_{k},$$
$$N_{i}=q^{}_{i}(N\ominus A)q^{}_{i},$$
$$P_{i}=\bigoplus_{i+1\leq k< \infty}
p^{}_{k}Np^{}_{k}\ominus\mathbb{C}p^{}_{k},$$ then the strong-operator closure of
$\bigcup^{\infty}_{i=1}N_{i}$ is $N\ominus A$. By Case $\rm{(i)}$ and Case $\rm{(iii)}$, we have
$$N_{i}\oplus P_{i}\subseteq\overline{\mathrm{span}\{u:u\in
\mathscr{U}(N \ominus A)\}^{\mathrm{SOT}}}.$$  Thus we finish the proof.
\end{proof}

\begin{lemma}
If $N$ is a type $\mathrm{II}_{1}$ factor and $A\subseteq N$ is an abelian von
Neumann subalgebra, then $N\ominus
A=\overline{\mathrm{span}\{u:u\in\mathscr{U}(N\ominus
A)\}^{\mathrm{SOT}}}$.
\end{lemma}
\begin{proof}
Since $A$ is an abelian von Neumann algebra, there exist two
mutually orthogonal projections $p,q\in A$ with sum $1$ such that
$pAp$ is a diffuse abelian von Neumann algebra with unit $p$ and
$qAq$ is an atomic abelian von Neumann algebra with unit $q$. Each
$x\in$ $N\ominus A$ can be written as
$$x=\begin{pmatrix}
x^{}_{11}&x^{}_{12}\\
x^{}_{21}&x^{}_{22}\\
\end{pmatrix}
$$
with respect to the decomposition $1=p+q$, where  $x^{}_{11}\in p(N\ominus A )p$, $x^{}_{12}\in pNq$,
$x^{}_{21}\in qNp$, $x^{}_{22}\in q(N\ominus A)q$.

Denote by $$N_{0}=\left\{\begin{pmatrix}
0&x^{}_{12}\\
x^{}_{21}&0\\
\end{pmatrix}:x^{}_{12}\in pNq,~~
x^{}_{21}\in qNp\right\}.$$ By Lemma $2.3$ and Lemma $2.4$, we only
need to prove
$$N_{0}\subseteq\overline{\mathrm{span}\{u:u\in\mathscr{U}(N\ominus
A)\}^{\mathrm{SOT}}}.$$

If $\tau(p)$ is rational, then the proof idea is the same with that we use in Case $\rm{(ii)}$
for $\tau(p)$ rational in Lemma $2.4$.

If $\tau(p)$ is irrational, then let $\{p^{}_{n}\}_{n\in
\Lambda}$ and $\{q^{}_{n}\}_{n\in\Lambda}$ be two families of increasing
subprojections of $p$ and $q$ respectively such that
$\tau(p^{}_{n})\rightarrow \tau(p)$, $\tau(q^{}_{n})\rightarrow
\tau(q)$ and for all $n\in \Lambda$, both $\tau(p^{}_{n})$ and
$\tau(q^{}_{n})$ are rational. Then for $n\in \Lambda$, $x\in N$, we have
$$p^{}_{n}xq^{}_{n} \xrightarrow{\mathrm{SOT}} pxq.$$ Next we show that
\begin{equation}
p^{}_{n}xq^{}_{n}\in\mathrm{span}\{u:u\in\mathscr{U}(N\ominus A)\}.
\end{equation}

For each $n\in \Lambda$, suppose
$$\frac{\tau(p_{n})}{\tau(q_{n})}=\frac{k^{}_{n}}{l^{}_{n}} \mbox{ with } k^{}_{n}, l^{}_{n}\in\mathbb{N^{+}}.$$
Let $\{p^{}_{n_{i}}\}_{1\leq i \leq k^{}_{n} }$ and $\{q^{}_{n_{j}}\}_{1\leq
j\leq l^{}_{n}}$ be two families of mutually orthogonal equivalent projections in $ N $ such that
$p^{}_{n_{1}}+p^{}_{n_{2}}+\cdots+p^{}_{n_{k_{n}}}=p^{}_{n}$,
$q^{}_{n_{1}}+q^{}_{n_{2}}+\cdots+q^{}_{n_{l_{n}}}=q^{}_{n}$. Then there exists a $*$-isomorphism $\varphi$ from
$N$ onto $\varphi(N)$ such that
$$\varphi((p^{}_{n}+q^{}_{n})N(p^{}_{n}+q^{}_{n}))=\mathbb{M}_{k^{}_{n}+l^{}_{n}}
(p^{}_{n_{1}}Np^{}_{n_{1}})$$ and
$\varphi|_{(1-p_{n}-q_{n})N(1-p_{n}-q_{n})}$ is the identity map.
Denote by
$$N^{}_{n_{0}}=\left\{\begin{pmatrix}
0&x\\
y&0\\
\end{pmatrix}: x\in p^{}_{n}Nq^{}_{n},~y\in q^{}_{n}Np^{}_{n}\right\},$$
$$N_{1}=((p-p_{n})N(p-p_{n}))\ominus ((p-p_{n})A(p-p_{n})),$$
$$N_{2}=((q-q_{n})N(q-q_{n}))\ominus ((q-q_{n})A(q-q_{n})),$$
$$M_{k^{}_{n}+l^{}_{n}}=\{\{x_{ij}\}^{}_{ 1\leq i,j\leq k^{}_{n}+l^{}_{n}}\in
\mathbb{M}^{}_{k^{}_{n}+l^{}_{n}}(p^{}_{n_{1}}Np^{}_{n_{1}}):
x^{}_{ii}=0,\textmd{ for } 1\leq i\leq k^{}_{n}+l^{}_{n}\}.$$

By Lemma $2.1$, we have
$$M_{k^{}_{n}+l^{}_{n}}=\mathrm{span}\{u:u\in\mathscr{U}(M_{k^{}_{n}+l^{}_{n}})\}.$$
By Lemma $2.3$ and Lemma $2.4$, there is a unitary operator $v\in
N_{1}\oplus N_{2}$. Note that $$\varphi(N^{}_{n_{0}})\subseteq
M_{k^{}_{n}+l^{}_{n}}\mbox{  and  } M_{k^{}_{n}+l^{}_{n}}\oplus v
\subseteq \varphi(N)\ominus\varphi(A),$$ so that
$$\varphi(N^{}_{n_{0}})\oplus
v\subseteq\mathrm{span}\{u:u\in\mathscr{U}(\varphi(N)\ominus\varphi(A))\},$$
thus $(2.9)$ holds.
\end{proof}

\begin{theorem}
If $N$ is a type $\mathrm{II}_{1}$ factor and $A\subseteq N$ is a type
$\mathrm{I}$ von Neumann subalgebra, then $N \ominus
A=\overline{\mathrm{span}\{u:u\in\mathscr{U}(N\ominus
A)\}^{\mathrm{SOT}}}$.
\end{theorem}
\begin{proof}
Since $A$ is a type $\mathrm{I}$ von Neumann algebra, there exists a
family of mutually orthogonal central projections
$\{p^{}_{i}\}^{}_{i\in \Lambda}\subseteq A$ with sum $1$ such that
$$p^{}_{i}Ap^{}_{i}\cong\mathbb{M}_{k_{i}}(A_{i}),$$ where for each
$i\in\Lambda$, $A_{i}$ is an abelian von Neumann subalgebra and
$k_{i}$ is some positive integer. So we assume
$$A=\bigoplus_{i\in\Lambda}\mathbb{M}_{k_{i}}(A_{i}),$$$$p^{}_{i}Np^{}_{j}=\mathbb{M}_{k_{i}\times
k_{j}}(N_{ij}),$$ where $N_{ij}=e_{i_{1}}Ne_{j_{1}}$,
$\{e_{i_{n}}\}_{1\leq n\leq k_{i}}$ is a family of mutually
orthogonal equivalent subprojections of $p^{}_{i}$ with sum
$p^{}_{i}$.

For each $i,j\in\Lambda,~i<j$, denote by
$$P_{ij}=\bigoplus_{k\in\Lambda,k\neq i,j}
p^{}_{k}(N\ominus A)p^{}_{k},$$
$$\widetilde{N}^{}_{ij}=(p^{}_{i}+p^{}_{j})(N\ominus A)(p^{}_{i}+p^{}_{j}),$$
$$\widehat{N}^{}_{ij}=\left\{\begin{pmatrix}
x^{}_{1}&x^{}_{2}\\
x^{}_{3}&x^{}_{4}\\
\end{pmatrix}:x^{}_{1}\in N_{ii}\ominus A_{i},x^{}_{2}\in N_{ij},x^{}_{3}\in
N_{ji},x^{}_{4}\in N_{jj}\ominus A_{j}\right\},$$$$\widetilde{S}_{ij}=\left\{\begin{pmatrix}
0&X^{}_{2}\\
X^{}_{3}&0\\
\end{pmatrix}:
X^{}_{2}\in p^{}_{i}Np^{}_{j}, X^{}_{3}\in
p^{}_{j}Np^{}_{i}\right\} ,$$$$\widehat{S}_{ij}=\left\{\begin{pmatrix}
X^{}_{1}&0\\
0&X^{}_{4}\\
\end{pmatrix}:
X^{}_{1}\in p^{}_{i}(N\ominus A)p^{}_{i},X^{}_{4}\in p^{}_{j}(N\ominus A)p^{}_{j}\right\}.$$

By Lemma $2.2$ and Lemma $2.5$, we have
\begin{equation}
p^{}_{k}(N\ominus
A)p^{}_{k}=\overline{\mathrm{span}\{u:u\in\mathscr{U}(p^{}_{k}(N\ominus
A)p^{}_{k})\}^{\mathrm{SOT}}},
\end{equation} so that
$$\widehat{S}_{ij}=\overline{\mathrm{span}\{u:u\in\mathscr{U}(\widehat{S}_{ij})\}^{\mathrm{SOT}}}.$$

Next we show
\begin{equation}
\widetilde{S}_{ij}\subseteq\overline{\mathrm{span}\{u:u\in\mathscr{U}(\widetilde{N}^{}_{ij})\}^{\mathrm{SOT}}}.
\end{equation}

For each $x=\{x^{}_{kl}\}^{}_{ 1\leq k,l\leq k_{i} +k_{j}}\in \widetilde{S}^{}_{ij}$, we have
$$\begin{pmatrix}
0&x^{}_{st}\\
x^{}_{ts}&0\\
\end{pmatrix}\in\widehat{N}^{}_{ij},$$ for $1\leq s\leq k_{i} ,~k_{i}+1\leq t \leq k_{i} +k_{j}$.
By Lemma $2.5$, we have
$$\widehat{N}^{}_{ij}=\overline{\mathrm{span}\{u:u\in\mathscr{U}(\widehat{N}
^{}_{ij})\}^{\mathrm{SOT}}}$$ and
$$(N_{ii}\ominus A_{i})^{(k_{i}-1)}=\overline{\mathrm{span}\{u:u\in\mathscr{U}((N_{ii}\ominus
A_{i})^{(k_{i}-1)})\}^{\mathrm{SOT}}},$$
so that each operator in
$$\widehat{N}^{}_{ij}\oplus(N_{ii}\ominus A_{i})^{(k_{i}-1)}\oplus(N_{jj}\ominus A_{j})^{(k_{j}-1)}$$
can be approximated in the strong-operator topology by a linear
combination of finitely many unitary operators in this set. Then
relation $(2.11)$ holds since
$$\widehat{N}^{}_{ij}\oplus(N_{ii}\ominus
A_{i})^{(k_{i}-1)}\oplus(N_{jj}\ominus
A_{j})^{(k_{j}-1)}\subseteq\widetilde{N}^{}_{ij}.$$ Thus we have
$$\widetilde{N}^{}_{ij}=\overline{\mathrm{span}\{u:u\in\mathscr{U}(\widetilde{N}^{}_{ij})\}^{\mathrm{SOT}}}.$$

If $\Lambda$ is a finite set, then by $(2.10)$, we have
$$P_{ij}=\overline{\mathrm{span}\{u:u\in\mathscr{U}(P_{ij})\}^
{\mathrm{SOT}}},$$ so that $$\widetilde{N}^{}_{ij}\oplus
P_{ij}=\overline{\mathrm{span}\{u:u\in\mathscr{U}(\widetilde{N}^{}_{ij}\oplus
P_{ij})\}^ {\mathrm{SOT}}},$$ which suffices to prove this theorem.

If $\Lambda$ is an infinite set, then
for each $i\in\Lambda$, denote by
$$q^{}_{i}=\sum^{i}_{ k=1} p^{}_{k},~~N_{i}=q^{}_{i}(N\ominus A)q^{}_{i},~~~P_{i}=\bigoplus_{i+1\leq k<\infty}
p^{}_{k}Np^{}_{k}\ominus p^{}_{k}Ap^{}_{k},$$  then
\begin{equation}
\overline{\bigcup_{i\in\Lambda}N_{i}^{\mathrm{SOT}}}=N\ominus A.
\end{equation}
According to the case where $\Lambda$ is finite, we obtain
$$N_{i}=\overline{\mathrm{span}\{u:u\in\mathscr{U}(N_{i})\}^
{\mathrm{SOT}}}.$$ By Lemma 2.2 and Lemma 2.5, we have
$$P_{i}=\overline{\mathrm{span}\{u:u\in\mathscr{U}(P_{i})\}^
{\mathrm{SOT}}},$$ so that\begin{equation}N_{i}\oplus
P_{i}=\overline{\mathrm{span}\{u:u\in\mathscr{U}(N_{i}\oplus
P_{i})\}^ {\mathrm{SOT}}}.\end{equation} Hence $(2.12)$ and $(2.13)$
allow us to complete this theorem.
\end{proof}


\bibliographystyle{amsplain}

\end{document}